\documentclass[12pt]{amsart}
\pdfoutput=1

\usepackage{amsmath, amsfonts, amssymb}
\usepackage{cellspace}
\usepackage{svg}
\usepackage{csquotes}
\usepackage[colorlinks=true, citecolor=cyan, urlcolor=magenta]{hyperref}
\usepackage{cleveref}

\makeatletter
 
 \@addtoreset{equation}{section}
\makeatother

\newtheorem{theorem}{Theorem}
\newtheorem{lemma}[theorem]{Lemma}
\newtheorem{proposition}[theorem]{Proposition}

\newtheorem{question}[theorem]{Question}

\newtheorem{corollary}[theorem]{Corollary}

\theoremstyle{remark}
\newtheorem{remark}[theorem]{Remark}
\newtheorem{example}[theorem]{Example}

\numberwithin{equation}{section}

\title{Connected Domination in Plane Triangulations}
\author{Felicity Bryant and Elena Pavelescu}
\begin{document}


\begin{abstract}

A set of vertices of a graph $G$ such that each vertex of $G$ is either in the set or is adjacent to a vertex in the set is called a dominating set of $G$. 
If additionally, the set of vertices induces a connected subgraph of $G$ then the set is a connected dominating set of $G$.  
The domination number $\gamma(G)$ of $G$ is the smallest number of vertices in a dominating set of $G$, and the connected domination number $\gamma_c(G)$ of $G$ is the smallest number of vertices in a connected dominating set of $G$.  
We find the connected domination numbers for all triangulations of up to thirteen vertices. 
For $n\ge 15$, $n\equiv 0$ (mod 3), we find  graphs of order $n$ and $\gamma_c=\frac{n}{3}$.
We also show that the difference $\gamma_c(G)-\gamma(G)$ can be arbitrarily large. 

\end{abstract}
\maketitle

\section{Introduction}

A graph $G$ is a structure defined by a set of vertices $V(G)$ and a set of edges $E(G)$ connecting the vertices. 
Two vertices connected by an edge are said to be \textit{adjacent}. 
A \textit{planar} graph is a graph that can be drawn in the plane in such a way that the edges intersect only at their endpoints. 
A planar graph is maximal, or a \textit{triangulation}, if any added edge creates a graph which is no longer planar. 
For a graph $G$, $\delta(G)$ denotes the minimum among the degrees of all vertices of $G$, and 
$\Delta(G)$ denotes the maximum among the degrees of all vertices.
For a vertex $v\in V(G)$,  $N(v)$ denotes the set of vertices of $G$ that are adjacent to $v$,  and $N\big[v\big]$ denotes $N(v)\cup \{v \} $.

A set of vertices $S\subset V(G)$ such that each vertex of $G$ is either in $S$ or is adjacent to a vertex in $S$ is a \textit{dominating set of $G$}. Additionally, if the induced subgraph $\big< S \big>_G$ is connected, then  $S$ is a \textit{connected dominating set of $G$.} 
The \textit{domination number} $\gamma(G)$ of a graph $G$ is the smallest number of vertices in a dominating
set of $G$. The \textit{connected domination number} $\gamma_c(G)$ of a graph $G$ is the smallest number of vertices in a
connected dominating set of $G$. Only connected graphs are considered in this article. Graphs with more than one connected component do not have connected dominating sets.

While the concept of dominating number was introduced as the \textit{coefficient of external stability} by Berge in \cite{Be}, the terms \textit{dominating set} and \textit{domination number} were coined by Ore in \cite{O}.
The concept of connected domination was first introduced by Sampathkumar and Walikar \cite{SW}.
Garey and Johnson \cite{GJ} showed that the problem of finding a minimum (connected) dominating set is NP-hard. 
This problem remains NP-complete even for planar graphs with maximum degree 3 \cite{GJ}. 
In \cite{K}, Karami et al  study connected domination in graphs and their complements and give upper bounds  for the sum $\gamma_c(G)+\gamma_c(\overline{G})$ and the product $\gamma_c(G)\cdot \gamma_c(\overline{G})$ in terms of the graph order. It is through \cite{K} that the authors were first made aware of connected domination of graphs.

The domination number of a triangulation has been shown to be at most a third of the number of vertices by Matheson and Tarjan \cite{MT}. 
The $n/3$ bound was improved to $5n/16$ by Plummer, Ye, and Zha \cite{PYZ2} for Hamiltonian plane triangulations and to $17n/53$ by Spacapan \cite{S} for all plane triangulations on more than 6 vertices.
Matheson and Tarjan conjectured that the $n/3$ bound could be improved to $n/4$ for sufficiently large order $n$.
This conjecture was confirmed for triangulations of maximal degree 6 by King and Pelsmajer \cite{KP}.
This article is about connected domination in plane triangulations. 
As far as the authors know, no similar results exist about the connected domination number of plane triangulations.
While the definitions imply that $\gamma(G)\le \gamma_c(G)$, an upper limit for the connected domination number in relation to the graph order has not been determined. In this article, we demonstrate that the Matheson and Tarjan conjecture does not extend to connected domination and we conjecture that for all graphs $G$ of order $n$, $\gamma_c(G)\le \frac{n}{3}$. We show this bound is attained by certain graphs of order $n=3k$, $k\ge 5$. 

We give examples of infinite families of triangulations for which the difference $\gamma_c(G)-\gamma(G)$ is arbitrarily large.
We also find connected domination numbers for plane triangulations up to thirteen vertices.

\section{Graphs with vertices of large degree}

In this section it is shown that the existence of a vertex  of large degree in the graph $G$ implies a small connected domination number $\gamma_c(G)$. A well known result is presented in the next Lemma.

\begin{lemma} For $G$ a connected graph with $n$ vertices,  $$\gamma_c(G)\le n- \Delta (G).$$
\label{Delta}
\end{lemma}
\vspace{-0.25in}
\begin{proof}
Let $v\in V(G)$ be such that $\deg(v)=\Delta(G)$. Construct a breadth-first spanning tree for $G$ rooted at $v$. This tree has at least $\Delta(G)$ leaves. It follows that $\gamma_c(G)\le n- \Delta (G),$ since the set of non-leaf vertices of the tree form a connected dominating set.   
\end{proof}

\begin{corollary}
Let $G$ be a triangulation with $n\geq 3$ vertices and $\Delta(G)=n-1$. Then $\gamma_c(G) = 1$
\end{corollary}

\begin{corollary}
Let $G$ be a triangulation with $n\geq 6$ vertices and  $\Delta(G)=n-2$. Then $\gamma_c(G) = 2$.
\end{corollary}

Lemma \ref{Delta} also implies that if $\Delta(G)=n-4,$ then $\gamma_c(G) \le 4$. For graphs with $10\le n\le 13$ vertices, this upper bound for $\gamma_c(G)$ can be improved.

\begin{lemma}
    Let G be a triangulation with $n\le 13$ vertices  and $\Delta(G) = \deg(v) = n-4$ for some vertex $v\in V(G)$. Then $\gamma_c(G) = 2$ or $3$.
    \label{lemma_n-4}
\end{lemma}

\begin{proof}
    Since $deg(v) = n-4$, $v$ is adjacent to all but three of the other vertices of $G$. Let $V(G)\setminus N\big[v\big] =\{u,w,t\}$.  Since there is no vertex of degree $n-1$, $\gamma_c(G)>1$.

    If all $v, u, w,$ and $t$ have a common neighbor $a$, then the set $\{a, v\}$ is a connected dominating set and $\gamma_c(G)=2$. See the figure on the left of Figure \ref{fig:lemma_n-4}.

     If  $v$ and two vertices in $\{u, w, t\}$, say $u$ and $w$, have a common neighbor $a$, then let $b$ be a common neighbor of $v$ and $t$. Such vertex $b$ exists because $\deg(t) \geq 3$. The set $\{a, b, v\}$ is a connected dominating set and $\gamma_c(G)\le 3$. See the figure on the right of Figure \ref{fig:lemma_n-4}.
    
\begin{figure}[h]
\centering
\includegraphics[width=0.8\textwidth]{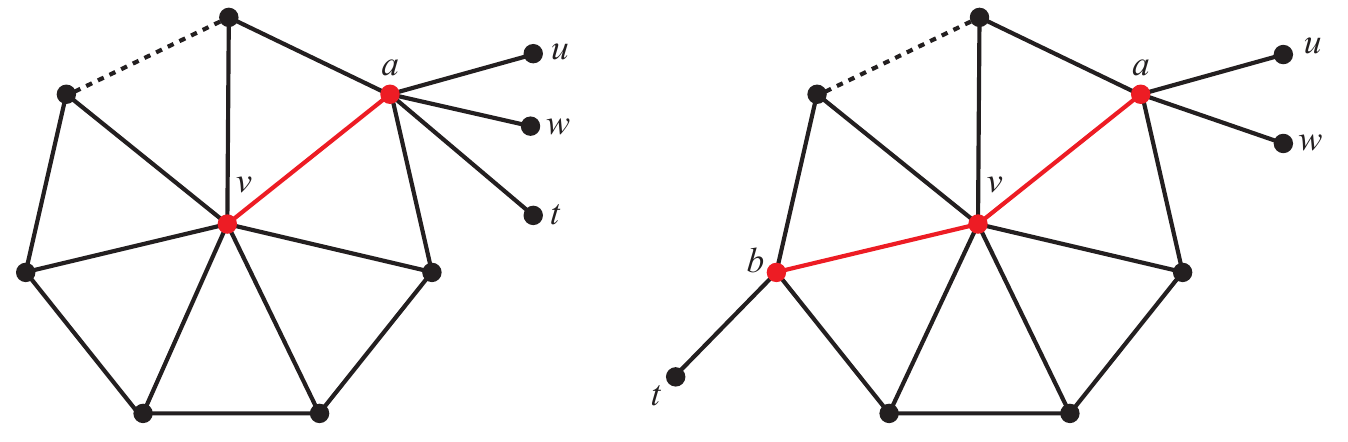}
\caption{\label{fig:lemma_n-4}Left: vertices $v, u, w$, and $t$ have a common neighbor $a$; Right: vertices $v, u, w$ have a common neighbor $a$, vertices $v$ and $t$ have a common neighbor $b$.}
\end{figure}

A common neighbor of $v$ and two vertices in $\{u, w, t\}$ occurs if the induced subgraph $\big<u, w, t\big>_G$ contains at least one edge. This edge, say $uw$, belongs to at least two triangles in $G$. 
One of these triangles can be $uwt$. The other triangle must be $uwa$, where $a$ is a neighbour of $v$. A common neighbor of $v$ and two vertices in $\{u, w, t\}$ also occurs if the induced subgraph $\big<u, w, t\big>_G$ contains no edge and $n \le 12.$ Since  $\deg(u), \deg(w), \deg(t)\ge 3,$ each $u,w,$ and $t$ shares at least three common neighbors with $v$. For $n\le 12$, by the pigeonhole principle, there exists a pair of vertices among $u, w$, and $t$ which share a common neighbor with $v$.

We assume $n=13$,  $\big<u, w, t\big>_G$  has no edges, and no pair of vertices in $\{u, w, t\}$ shares a common neighbor with $v$. Since $\deg(u), \deg(w), \deg(t)\ge 3,$ each $u,w,$ and $t$ shares at least three common neighbors with $v$.  Since $n=13$, each vertex $u, w,$ and $t$ shares exactly three common neighbors with $v$ and these three sets of three vertices are pairwise disjoint. It follows that the graph $G$ contains the graph on the left of Figure \ref{fig:lemma13}.
A triangulation is obtained by adding edges between the six vertices highlighted in blue. Without loss of generality, edge $e$ added on the right of Figure \ref{fig:lemma13} belongs to $G$, and the set $\{a,b,c\}$ is a connected dominating set of $G$. 

\begin{figure}[h]
\centering
\includegraphics[width=0.9\textwidth]{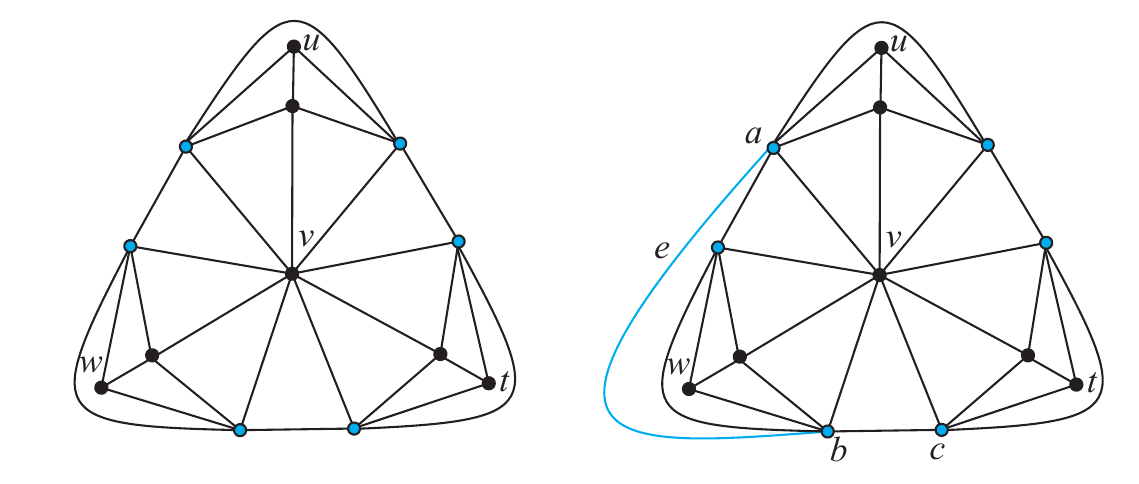}
\caption{\label{fig:lemma13}Left: Subgraph of $G$.  Right: The set $\{a,b,c\}$ is a connected dominating set of $G$ .}
\end{figure}

\end{proof}

\begin{remark}
   The result in Lemma \ref{lemma_n-4} does not hold for $n\ge 14$. The graph $H$ in Figure \ref{fig:remark14} has order 14, maximum degree 10, and $\gamma_c(H)=4$. The set $\{ a, c, e, v\}$ is a connected dominating set of $H$. We see that $\gamma_c(H)> 3$ as follows.
   A connected dominating set contains at least one neighbor of each $u,w,t,$ and $p$. With these constraints, if $\gamma_c(H)=3$, the vertex dominating set contains $e$  and one vertex in each $\{a,b\}$ and $\{c,d\}$.  None of the four choices gives a connected subgraph of $H$.
\end{remark}

\begin{figure}[h]
\centering
\includegraphics[width=0.43\textwidth]{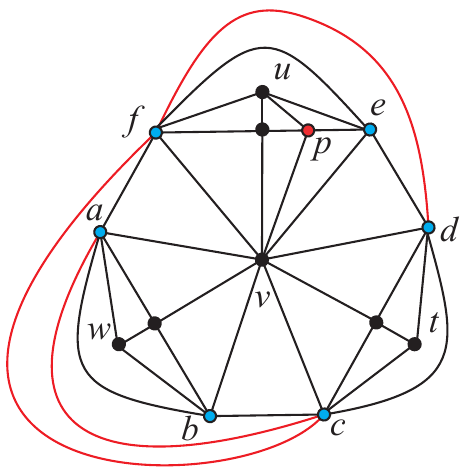}
\caption{\label{fig:remark14} Graph of order 14 with $\gamma_c=4$.}
\end{figure}

\section{Connected domination in triangulatios of up to fifteen vertices}

Plantri \cite{BM}, Nauty Traces \cite{MP} and Mathematica were used to determine connected domination numbers for all graphs up to order 13 and to show that all graph of order 14 have connected domination numbers at most 4. Connected domination numbers were found by hand for all triangulations of up to eight vertices.
For each $9\le n\le 14$, Plantri was used to generate all triangulations of order $n$, then Nauty Traces was used to sort the triangulations into those with maximal degree $n-1$ (these have $\gamma_c=1$ by Lemma \ref{Delta}), those with maximal degree $n-2$ (these have $\gamma_c=2$ by Lemma \ref{Delta}), and those with $3\le \deg(v) \le n-3$ for all vertices $v$ of the graph. 

The connected domination number of each graph in the third category  was checked using a Mathematica program which takes successive edge contractions of a graph.
An edge contraction means identifying the endpoints of the edge to a vertex and deleting all double edges thus created. The new vertex inherits all neighbors of the endpoint vertices in the original graph. For each graph $G$ of order $n$, the program checked if there exists an one-edge contraction which creates a vertex of degree $n-2$. If so, this graph has $\gamma_c=2$, with the endpoints of the edge representing the dominating set of cardinality two.
Else, the program checked if there exist two one-edge contractions which create a vertex of degree $n-3$. The choice of two edges was not restricted to pairs sharing an endpoint. Nevertheless, if a vertex of degree $n-3$ was found, then the graph has $\gamma_c=3$. This is because, as seen in  Lemma \ref{lemma:contraction},  the set of edges that were contracted can be assumed to form a connected subgraph of $G$. The program stopped when a vertex of maximal degree was found through repeated contractions. 

\begin{lemma}
Let $G$ be a graph of order $n$ and let $S$ be a set of $k$ edges of $G$ such that the minor of $G$ obtained by contracting all edges in $S$ contains a vertex of degree $n-k-1.$ Then there exists a set $T$ of $k$ edges of $G$ such that the minor of $G$ obtained by contracting all edges in $T$ contains a vertex of degree $n-k-1$, and the edges in $T$ form a connected subgraph of $G$.
\label{lemma:contraction}
\end{lemma}
\begin{proof}
 The proof is by induction on the number $k$ of contracted edges.
 The basis step $k=0$ is immediate.
 Assume the statement holds when $k$ edges are contracted.
 Consider the graph $G$ such that $k+1$ edge contractions of $G$ create a vertex of degree $n-k-2$. 
 Call these edges $a_1b_1$, $a_2b_2, \ldots, a_{k+1}b_{k+1}$. 
 Consider the graph $G'=G/a_1b_1$ obtained from $G$ by contracting the edge $a_1b_1.$
  Call $v$ the vertex of $G'$ obtained from contracting $a_1b_1$. 
 By induction, $G'$ has a set $T$ of $k$ edges such that the minor of $G'$ obtained by contracting all edges in $T$ contains a vertex of degree $n-k-2$, and the edges in $T$ form a connected subgraph of $G'$.
 If $v$ is an endpoint of one of the edges in $T$, then $a_1b_1$ together with the edges in $T$ form the desired set of edges of $G$.
 If $v$ is not an endpoint of any edge in $T$, then $v$ is adjacent to at least one endpoint of an edge in $T$. 
 Call this vertex $u$. 
 This implies either $a_1$ or $b_1$ is adjacent to $u$ in $G$.
 Assume $a_1u\in E(G).$
 Then $T\cup \{a_1u\}$ is the desired set of edges.
\end{proof}

The computer searches gave the values presented in Table \ref{table9-13}. 
For 14 vertices, the number of graphs was too large to handle by computer.  Nevertheless, they all have connected domination at most 4, as shown below.

\begin{proposition}
All triangulations $G$ with 14 vertices have $\gamma_c(G)\le 4$.
\end{proposition}
 
\begin{proof}
By Lemma \ref{Delta}, if $\Delta(G)\ge 10$, then $\gamma_c(G)\le 4$. 
Consider $G$ a graph with $\delta(G)=3$.
Then $G$ is obtained from a triangulation $H$ with $n=13$ vertices by addition of one vertex within a face of an embedding of $H$ as in the first diagram of Figure \ref{fig:Tmoves}. Adding one vertex can increase $\gamma_c$ by at most 1. This increase occurs when none of the neighbours of the new vertex pertains to a minimal connected dominating set of $H$. If $\gamma_c(H)\le 3$, then $\gamma_c(G)\le 4$.  
There are 153 graphs $H$ with 13 vertices and $\gamma_c(H)=4$. It is possible that, by adding a vertex of degree three to one of these graphs, one can obtain a graph $G$ with $\gamma_c(G)=5$. To check  if this is the case, using Nauty, a vertex of degree 3 was added every which way to the 153 graphs. Up to isomorphism, 2,445 triangulations were found this way. Using Mathematica, each of these graphs was shown to have $\gamma_c(G)=4.$

With the considerations above, only the triangulations $G$ with $\delta(G)\ge 4$ and $\Delta(G)\le 9.$ need to be checked. There are 2,231 such graph. Using Mathematica, each of these graphs was shown to have $\gamma_c(G)\le 4.$
\end{proof}

\begin{proposition}
All triangulations $G$ with 15 vertices have $\gamma_c(G)\le 5.$
\end{proposition}

\begin{proof}
By work of Bowen and Fisk \cite{BF}, a  triangulation with $n$ vertices can be obtained from a triangulation with $n-1$ vertices by adding one vertex of degree 3, 4, or 5, as described in Figure \ref{fig:Tmoves}. If $G'$ is obtained from $G$ through either of the three moves, then $\gamma_c(G')\le \gamma_c(G)+1.$ 
A triangulation with $15$ vertices is constructed from a triangulation with $14$ vertices by addition of one vertex. 
Since $\gamma_c(G)\le 4$ for all triagulations $G$ with $14$ vertices, then $\gamma_c(G')\le 5$ for all triagulations $G'$ with $15$ vertices. 
\end{proof}
\begin{figure}[h]
\centering
\includegraphics[width=1\textwidth]{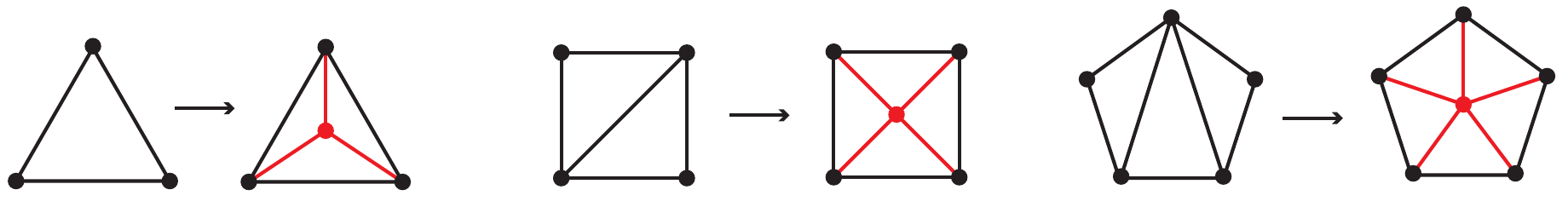}
\caption{\label{fig:Tmoves} Constructing larger triangulations from smaller ones. }
\end{figure}

\begin{remark}
We see from Table \ref{table9-13} that for each $9\le n\le 13$, for a graph $G$ of order $n$, $\gamma_c(G)\le \lfloor \frac{n}{3}\rfloor$. We also see that the conjecture of  Matheson and Tarjan \cite{MT} that $\gamma(G)\le \lfloor \frac{n}{4}\rfloor$ for triangulations of order $n$ does not extend to connected domination.
\end{remark}

\begin{table}
\centering
\begin{tabular}{|c|c|c|c|c|c|c|}
\hline
    $n$     & \#T&  $\gamma_c$=1  & $\gamma_c$=2  & $\gamma_c$=3 & $\gamma_c$=4 & $\gamma_c$=5 \\ \hline 
        $5$     & 1 & 1  & 0 & 0 & 0 & 0 \\ \hline
    $6$     & 2 &  1 & 1 & 0 & 0 & 0\\ \hline
     $7$     & 5 & 3 & 2 & 0 & 0 & 0\\ \hline
     $8$     & 14 & 3  & 11 & 0 & 0 & 0\\ \hline
    $9$     & 50 &  12 & 37 & 1 & 0 & 0\\ \hline
    $10$     & 233 & 27  & 193 & 13 & 0 & 0\\ \hline
     $11$     & 1,249 & 82  & 995  & 172 & 0 & 0\\ \hline
        $12$     & 7,595 & 226  & 5,191 & 2,173 & 5 & 0\\ \hline
        $13$     & 49,566 & 733  & 25,760 & 22,920 & 153 & 0 \\ \hline
    $14$    & 339,722 & 2,282 & - & - & - & 0\\ \hline
    $15$    & 2,406,841 & 7,528 & - & - & - & $>0$\\ \hline
\end{tabular}\vspace{0.1in}
\caption{\small Connected domination numbers for graphs of order up to 15. For $n\le 14$ no graph has $\gamma_c >4$.}
\label{table9-13}        
\end{table} 
\section{Two families of graphs}
Though their number is yet unknown, there exist graphs of order 15 with $\gamma_c=5$. They correspond to the ``$>0$" entry in Table \ref{table9-13}. These examples are presented in this section together with two families of graphs on $n$ vertices with $\gamma_c=\frac{n}{3}$ for $n\equiv 0$ (mod 3), $n\ge 15$. These examples were obtained using the construction in this next Lemma.
\begin{lemma}
Consider a triangulation $G$ and a triangle $abc$ of $G$ which bounds a face in a planar embedding of $G$.  Consider the triangulation $G'$ obtained from $G$ by adding a triangle $efg$ together with edges $ae, af, bf, bg, ce,$ and $cg$. 
(1) If only one of $a, b$, or $c$ is part of a minimal connected dominating set of $G$, then $\gamma_c(G')=\gamma_c(G)+1.$
(2) If none of $a, b$, or $c$ is included in a minimal connected dominating set of $G$, then $\gamma_c(G')=\gamma_c(G)+2.$
\label{lemma:twist}
\end{lemma}
\begin{proof}
 Since vertices $e$, $f$, and $g$ have no common neighbor in $G'$, a minimal connected dominating set of $G'$ contains at least two vertices which neighbor either $e$, $f$, or $g$. See Figure \ref{fig:twist3}. These two vertices belong to the set $\{a,b,c,e,f,g\}$. If only one of $a, b$, and $c$ is part of a minimal connected dominating set of $G$, then a minimal connected dominating set of $G'$ contains one more element than a minimal connected dominating set of $G$, and $\gamma_c(G')=\gamma_c(G)+1.$ If none of $a, b$, and $c$ is included in a  minimal connected dominating  set of $G$, then a minimal connected dominating set of $G'$ contains two more elements than a minimal connected dominating set of $G$, and $\gamma_c(G')=\gamma_c(G)+2.$
\end{proof}

\begin{figure}[ht]
\centering
\includegraphics[width=0.4\textwidth]{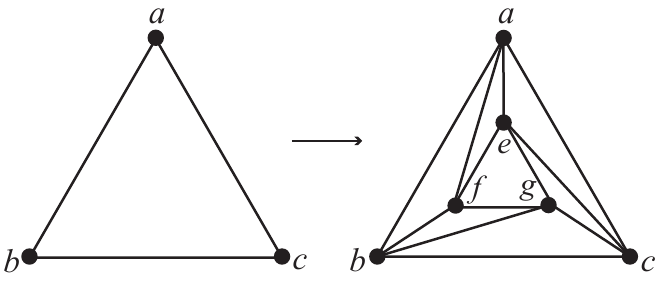}
\caption{\label{fig:twist3} Triangulation of order $n+3$ obtained from a triangulation of order $n$.}
\end{figure}

\newpage

\begin{remark} The construction in Lemma \ref{lemma:twist} can be thought of as a clique sum over the complete graph $K_3$ between the triangulation $G$ and the 4-regular triangulation of the plane with six vertices.
    \end{remark}

\begin{proposition}
There exists families $\{H_{3k}\}_{k\ge 5}$ of triangulations on $n=3k$ vertices, such that  $\gamma_c(H_{3k})= k = \frac{n}{3}$.
\label{prop:3kexamples}
\end{proposition}
\begin{proof}
Consider the triangulation on nine vertices on the left of Figure \ref{fig:3k}. This graph has $\gamma_c=2$ and the two highlighted vertices are a dominating set. No vertex of the triangle bounding the unbounded face of the embedding is contained in a minimal connected dominating set. With Lemma \ref{lemma:twist}, it follows that the graph in the middle of Figure \ref{fig:3k} has $\gamma_c=4.$
A minimal connected dominating set is highlighted.
No minimal connected domination set of this graph with 12 vertices contains two vertices of the triangle bounding the unbounded face.
\begin{figure}[ht]
\centering
\includegraphics[width=1\textwidth]{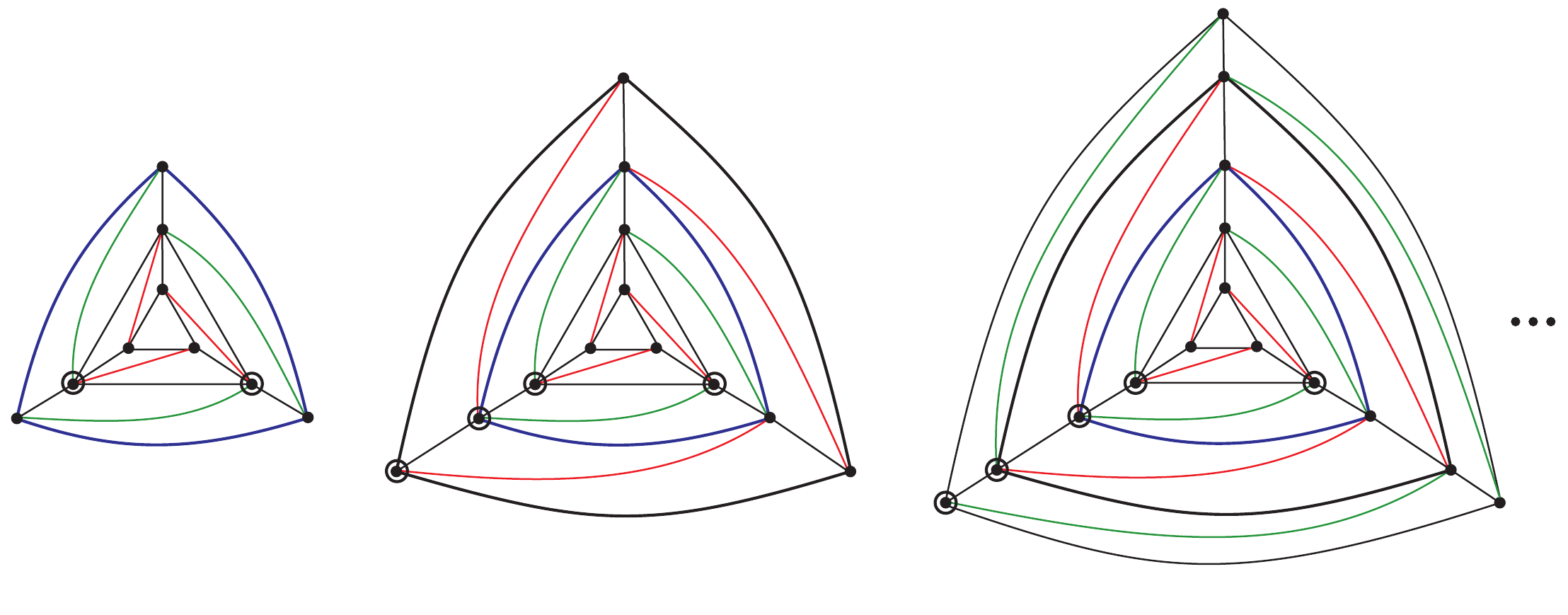}
\put(-400,0){$n=9, \gamma_c=2$ \hspace{0.6in} $n=12, \gamma_c=4$  \hspace{1.1in} $n=15, \gamma_c=5$}
\caption{\label{fig:3k} Repeated clique sums over $K_3$ yield graphs of order $3k$ and $\gamma_c=k,$ $k\ge 5.$}
\end{figure}
By Lemma \ref{lemma:twist}, the graph on the right of Figure \ref{fig:3k} has $\gamma_c\ge 5$. This is the graph $H_{15}$. A connected dominating set is highlighted, thus $\gamma_c(H_{15})=5.$ As above, no two vertices of the triangle bounding the unbounded face of $H_{15}$ belong to a minimal connected dominating set, and the construction can be iterated indefinitely to create $H_{3k}$, $k\ge 5$.

A second family of graphs can be obtained by starting with the graph on the left of Figure \ref{fig:3k2} and repeatedly performing clique sums over $K_3$ as in Lemma \ref{lemma:twist}. 
\begin{figure}[ht]
\centering
\includegraphics[width=1\textwidth]{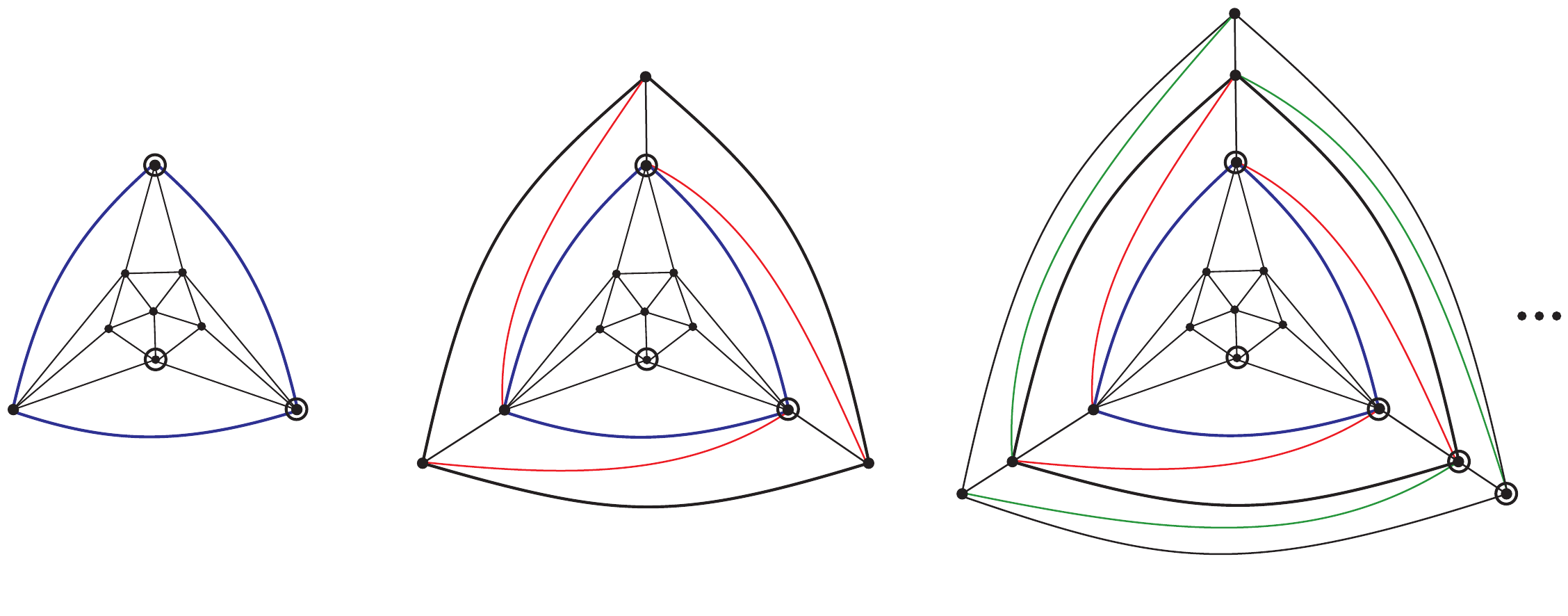}
\put(-400,0){$n=9, \gamma_c=3$ \hspace{0.7in} $n=12, \gamma_c=3$  \hspace{1.1in} $n=15, \gamma_c=5$}

\caption{\label{fig:3k2} Repeated clique sums over $K_3$ yield a second family of graphs of order $3k$ and $\gamma_c=k$, $k\ge 5.$}
\end{figure}
\end{proof}

Based on all our findings above we put forth the following question.

\begin{question}
Is it true that $\gamma_c(G)\le \frac{n}{3}$ for all   triangulations with $n\ge 3$ vertices?
\end{question}

\section{The difference $\gamma_c(G)-\gamma(G)$}

\begin{example}
    Among all triangulations with at most nine vertices, only one has its connected domination number different than its domination number. This same graph is the only one with $\gamma_c\ge 3$ among all the graphs of order at most nine. This is the graph $G_9$ in Figure \ref{fig:example9}.  For the graph $G_9$, $\gamma(G)=2$, given by the vertex set $\{u,v\}$, and $\gamma_c(G_9)=3$, given by the vertex set  $\{u,v, w\}$. This graph gives rise to the second family of graphs presented in Proposition \ref{prop:3kexamples}.
    \label{example9}
\end{example}

\begin{figure}[ht]
\centering
\includegraphics[width=0.33\textwidth]{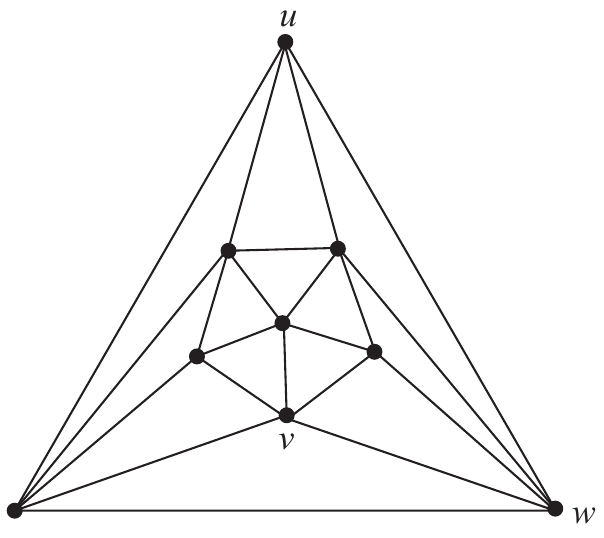}
\caption{\label{fig:example9} Triangulation $G_9$ of order nine with $\gamma_c=3$ and $\gamma=2$. }

\end{figure}

\begin{example}
    Among all triangulations with $n\le 12$ vertices, only two triangulations $G$ are such that $\gamma_c(G)>\gamma(G)+1$. These are the two graphs presented in Figure \ref{fig:example_icos}. The graph on the left is the icosahedron graph.
Both these graphs have $\gamma=2$, given by the set of vertices $\{a,u\}$ and $\gamma_c=4$. 
One can check that no connnected set of three vertices dominates the graph.
For each example, there are many possible choices of sets of four vertices which give $\gamma_c=4$.
One such choice, together with the connected subgraph they induce is highlighted in red.
\label{example12}
\end{example}
\begin{figure}[h]
\centering
\includegraphics[width=0.8\textwidth]{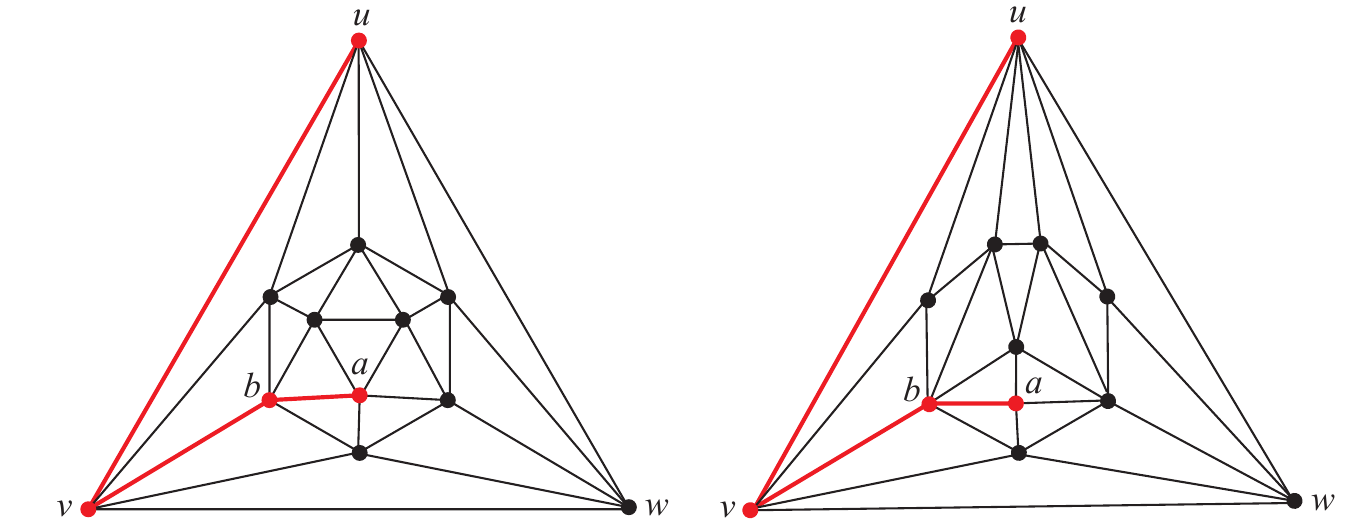}
\caption{\label{fig:example_icos} Triangulations  of order 12 with $\gamma_c=4$ and $\gamma=2$. }
\end{figure}

Examples \ref{example9} and \ref{example12} indicate the possibility that $\gamma_c-\gamma$ can become large for graphs of large order. Indeed, there are families of graphs for which $\gamma_c-\gamma$ grows infinitely large.
\begin{proposition}
There exists a family of triangulations of the plane $\{H_k\}_k$ for which $\gamma_c(H_k)-\gamma(H_k)$ grows arbitrarily large. 
\end{proposition}

\begin{proof}
Consider $k$ copies, $G_1, G_2, \ldots, G_k$, of  the icosahedron graph presented on the left in Figure \ref{fig:example_icos}.
Label the three vertices on the outer face of  $G_i$ by  $u_i$, $v_i$, and $w_i$. Following the notation in Example \ref{example12}, let $\{a_i,u_i\}$ be a dominating set of vertices for $G_i$ and let $\{a_i, b_i, u_i, v_i\}$  be a connected dominating set of vertices for $G_i$. 
Consider the graph obtained by gluing the edge $(u_1,v_1)$ of $G_1$ to the edge $(u_2, v_2)$ of $G_2$, then $(u_i, w_i)$ of $G_i$ to the edge $(u_{i+1}, v_{i+1})$ of $G_{i+1}$, for $i=2,\ldots, k-1$. In particular, all vertices $u_i$, $i=1,\ldots,k$, are identified to one vertex $u$. The graph $H_k$ is obtained by adding edges within the outside face of this graph until a triangulation is obtained. For $k\ge 3$, these edges can be added in more than one way. See Figure \ref{fig:example_icos2} for one example.

The set $\{u, a_1, a_2, \ldots, a_k\}$ is a dominating set for $H_k$. Since  $\gamma(G_i)=2$, for $i=1,
\ldots, k$, a dominating set for $H_k$ necessarily contains two vertices of each $G_i$. The set $\{u_i, v_i, w_i\}$ is not a dominating set for $G_i$, thus a dominating set of $H_k$ necessarily contains a vertex in $V(G_i) \setminus \{u_i, v_i, w_i\}$, for each $i=1,2,\ldots, k$.  This shows that the set $\{u, a_1, a_2, \ldots, a_k\}$ is a smallest dominating for $H_k$ and $\gamma(H_k)=k+1.$ 

The set $S=\{u, a_1, b_1, v_1=v_2, a_2, b_2, w_2=v_3, \ldots,a_{k-1}, b_{k-1}, w_{k-1}=v_k,a_k, b_k\}$ is a connected dominating set for $H_k$. A connected dominating set for $H_k$ necessarily contains four vertices of each $G_i$. Since at least two of these vertices are not in the set $\{u_i, v_i, w_i\}$, the set $S$ is a smallest connected dominating set for $H_k$ and $\gamma_c(H_k) = 3k$. The difference $\gamma_c(H_k)-\gamma(H_k)=3k-(k+1)=2k-1$ grows infinitely large with $k$.

\end{proof}

\begin{figure}[h]
\centering
\includegraphics[width=1\textwidth]{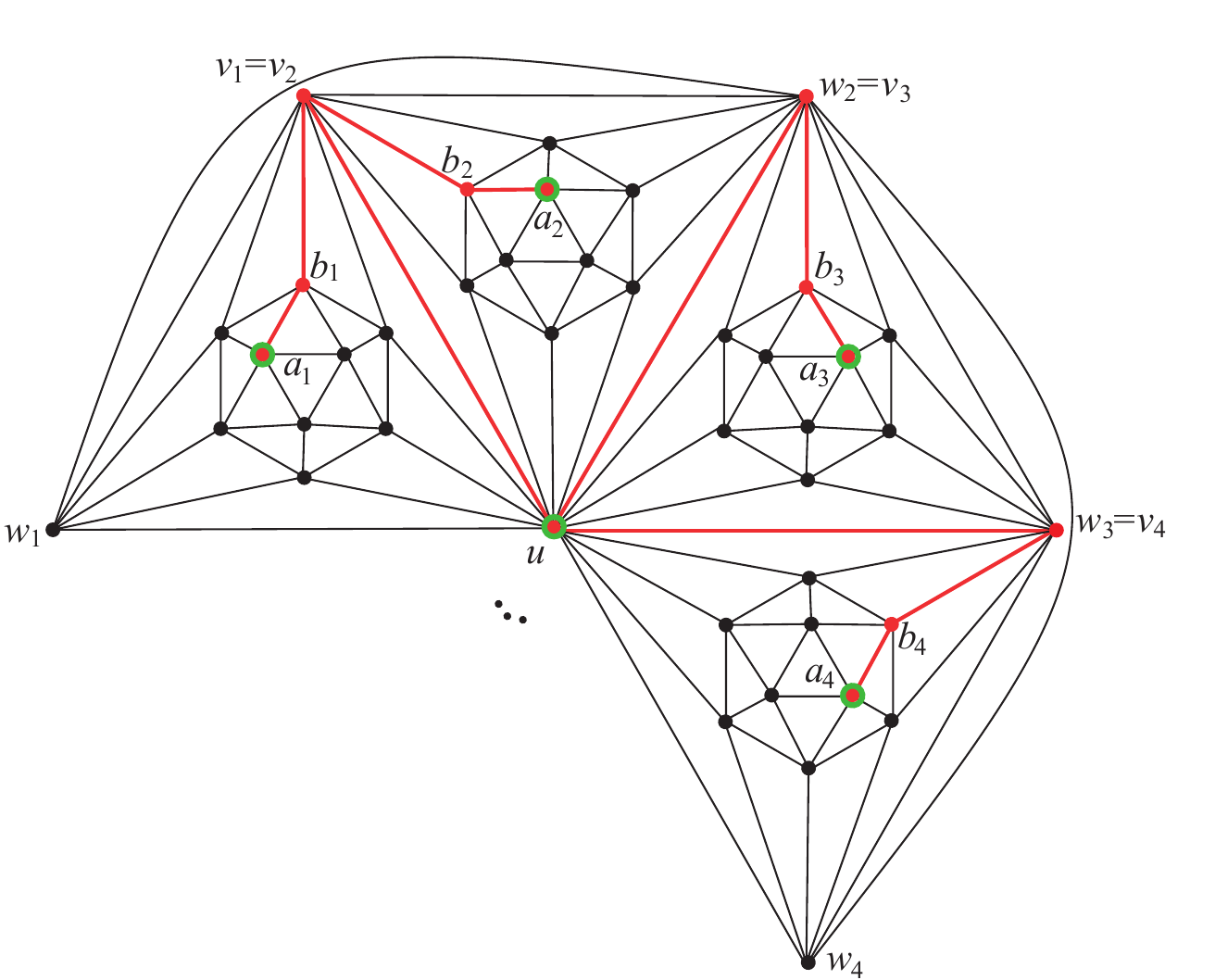}
\caption{\label{fig:example_icos2} A family of graphs for which $\gamma_c(H_k)-\gamma(H_k)$ grows arbitrarily large.  }

\end{figure}

\begin{remark}
    Using the same gluing pattern, the second graph with 12 vertices presented in Example \ref{example12} can be used to construct another family of graphs $\{J_k\}_k$. This family will have the same domination numbers and connected domination numbers as the $\{H_k\}_k$ family. 
\end{remark}
\pagebreak
\section{acknowledgements}
This research was supported in part by an MAA Tensor Grant for Women and Mathematics, together with the Summer Undergraduate Research Fellowship from the University of South Alabama, and the Alabama Space Grant Consortium.


\end{document}